\documentclass[11pt]{amsart}
\usepackage{geometry}                
\usepackage{graphicx}
\usepackage{amssymb}
\usepackage{epstopdf}
\newcommand\sphere[1]{\mathbb{S}^{#1}}
\newcommand\innerproduct[2]{\langle #1,#2\rangle}
\newcommand\norma[1]{\|#1\|}
\newtheorem{thm}{Theorem}[section]

\newtheorem{lem}[thm]{Lemma}

\theoremstyle{definition}
\newtheorem{definition}[thm]{Definition}

\title{A characterization of zonoids}
\author{Yossi Lonke}
\date{February 2020}

\begin{document}
\maketitle
\section{Introduction}
A function defined on the $n$-dimensional sphere $\sphere{n-1}$ will be called \emph{zonal} if its value at a point $x\in\sphere{n-1}$ depends only on the
angle between $x$ and a fixed axis. Thus if $u$ is a unit vector in the direction of the fixed axis, then a function is zonal with respect to $u$ if its value at $x$
depends only on the standard inner product $\innerproduct{x}{y}$. A natural way to generate zonal functions is as follows. Let $f\in C(\sphere{n-1})$, and fix a direction ${u\in\sphere{n-1}}$.
Let $(O(u),m)$ denote the subgroup of orthogonal transformations which keep the point $u$ fixed, equipped with the normalized Haar measure, $m$. Define
$$(S_uf)(x)=\int_{O(u)}f(Tx)\,dm(T)$$
Clearly, the function $S_uf$ is zonal. Appyling this procedure to support functions leads to the following definition.
\begin{definition}
Suppose $u\in\sphere{n-1}$ and $K$ is a centrally symmetric convex body, with support function $h_K$. The $u$-spin of $K$ is the convex body $K_u$ whose support function is $S_uh_K$.
\end{definition}
For example, computing the $e_n$-spin of the unit cube in $\mathbb{R}^n$, gives:
$$(S_{e_n}h_{B_n^{\infty}})(x)=\int_{O(e_n)}\norma{Tx}_1\,dm(T)=c_n\left(\sum_{i=1}^nx_i^2\right)^{1/2}+|x_n|,$$
($c_n$ is a positive number depending on $n$), which is the support function of a cylinder. This illustrates the choice of the word 'spin'.

Thus, with each body there is associated a system of rotation-bodies, one for each possible direction. Since $S_uf(u)=f(u)$ for each ${u\in\sphere{n-1}}$ and every function $f$, a body
is seen to be uniquely defined by its spins. The main result of this note is the following.

\begin{thm} A centrally symmetric convex body is a zonoid if and only if all its spins are zonoids.
\end{thm}
\section{Preliminaries}
$C^{\infty}(\sphere{n-1})$ is the Frechet space of functions on $\sphere{n-1}$ that have derivatives of every order. Elements of this space are called \emph{test functions}. Its dual space, denoted
${\mathcal D}(\sphere{n-1})$, is the space of distributions on $\sphere{n-1}$. If a subscript $e$ is added to any of the above spaces, it is to designate the
subspace of even objects.

The \emph{cosine transform} is the operator ${\mathcal C}:C_e^{\infty}(\sphere{n-1})\to C_e^{\infty}(\sphere{n-1})$ defined by:
$$({\mathcal C}f)(x)=\int_{\sphere{n-1}}|\innerproduct{x}{y}|f(y)\,d\sigma_{n-1}(y),$$
where $\sigma_{n-1}$ is the normalized rotation-invariant measure on the sphere. 
It is well known that ${\mathcal C}$ is a continuous bijection of $C_e^{\infty}(\sphere{n-1})$ onto itself, and that it can be extended by duality
to a bi-continuous bijection of the dual space ${\mathcal D}_e(\sphere{n-1})$. Hence,
if $\rho$ is a distribution and ${f\in C_e^{\infty}(\sphere{n-1})}$ is a test function, then 
$$\innerproduct{{\mathcal C}\rho}{f}=\innerproduct{\rho}{{\mathcal C}f}$$
Since $C_e^{\infty}(\sphere{n-1})$ and its dual space, the even measures, are both naturally embedded in ${\mathcal D}_e(\sphere{n-1})$, it makes sense to speak about the cosine transform of a measure,
or of a continuous function. A fundamental connection between distributions and centrally symmetric convex bodies was discovered by Weil, in \cite{weil76}. Weil proved that for every centrally symmetric convex
body $K\subset\mathbb{R}^n$ there corresponds a unique distribution $\rho_K$ (called the \emph{generating distribution} of $K$) such that ${\mathcal C}\rho_K=h_K$, where $h_K$ is the support function of $K$.
Suppose $f\in C^{\infty}(\sphere{n-1})$. Then for every direction $u\in\sphere{n-1}$, the function $S_uf$ also belongs to $C^{\infty}(\sphere{n-1})$. Therefore, $S_u$ can be defined to act on distributions by duality:
$$\innerproduct{S_u\rho}{f}=\innerproduct{\rho}{S_uf},\quad \rho\in{\mathcal D}(\sphere{n-1}),\ \ f \in C^{\infty}(\sphere{n-1})$$
A routine verification shows that the transforms $S_u$ and ${\mathcal C}$ commute on test functions, and therefore also as transforms of distributions.
\section{Proof of Theorem 1.1}
If $K$ is a zonoid,  then $h_K={\mathcal C}\mu$ for some positive measure $\mu$, and for every $u\in\sphere{n-1}$,
$$S_uh_K=S_u({\mathcal C}(\mu))={\mathcal C}(S_u\mu)$$
Since $S_u\mu$ is a positive measure for every $\mu$, every spin of $K$ is a zonoid. This proves the easy part of the theorem.

Suppose $K$ is a centrally symmetric convex body every spin of which is a zonoid. That is, for each direction $u\in\sphere{n-1}$ there exists a positive measure $\mu_u$ such that $S_uh_K={\mathcal C}\mu_u$.
There exists a distribution $\rho$ such that $h_K={\mathcal C}\rho$, and so the assumption reduces to ${\mathcal C}(S_u\rho)={\mathcal C}\mu_u$, where the commuting of $S_u$ and ${\mathcal C}$
was used. Since ${\mathcal C}$ is one-to-one, the distribution $\rho$ is seen to satisfy $S_u\rho=\mu_u$ for every $u\in \sphere{n-1}$. It therefore remains to prove:
\begin{lem} A distribution is positive if and only if $S_u\rho$ is positive for every $u\in\sphere{n-1}$.
\end{lem}
\begin{proof} The "only if" part is obvious. Suppose $\rho\in{\mathcal D}(\sphere{n-1})$ is a distribution such that $S_u\rho$ is positive for every $u\in\sphere{n-1}$. Then $\innerproduct{\rho}{f}\geq 0$ for every positive
zonal test function $f$. Therefore if $g=\sum_1^ma_if_i$, where $a_i\geq 0$ and $f_i$ are positive zonal test functions, then also $\innerproduct{\rho}{g}\geq 0$.

Choose a positive test function $f\in C^{\infty}(\sphere{n-1})$ and write its Poisson integral:
$$P_rf(x)=\int_{\sphere{n-1}}\frac{1-r^2}{\norma{x-ry}^n}f(y)\,d\sigma_{n-1}(y)$$
It is well know that $\lim_{r\to 1^-}P_rf=f$ in the topology of $C^{\infty}(\sphere{n-1})$. Therefore, if $\innerproduct{\rho}{P_rf}\geq 0$ for every $0<r<1$ and every positive test function $f$, then $\rho$ is positive, by continuity.
Fix a sequence of convex combinations of Dirac measures of the form 
\begin{equation}\label{approx}
\nu_N=\sum_{i\geq 1}\lambda_{iN}\delta_{y_i},\quad y_i\in\sphere{n-1}
\end{equation}
such that $\nu_N\to\sigma_{n-1}$ in the $w^*$ topology of measures. If $\varphi$ is a test function, then
\begin{equation}\label{eq2}
\int_{\sphere{n-1}}\frac{1-r^2}{\norma{x-ry}^n}\varphi(y)\,d\nu_N(y)\to \int_{\sphere{n-1}}\frac{1-r^2}{\norma{x-ry}^n}\varphi(y)\,d\sigma_{n-1}(y)=P_r\varphi(x)
\end{equation}
If $F(x,y)$ is a continuous function on $\sphere{n-1}\times\sphere{n-1}$ then for every test function $\varphi$, the integrals $\innerproduct{F(x,y)\varphi(y)}{\nu_N}$ converge to the integral  $\innerproduct{F(x,y)\varphi(y)}{\sigma_{n-1}}$
uniformly in $x$. Therefore, the sequence of functions of the variable $x$ in the l.h.s of (\ref{eq2}) converge to the r.h.s in the topology of $C^{\infty}(\sphere{n-1})$, because the function $\norma{x-ry}^{-n}$ for
fixed $y$ and $0<r<1$ is $C^{\infty}$ with respect to $x$. Moreover, the distribution $\rho$ is positive on every term of the l.h.s of (\ref{eq2}), because for each $y_i$ appearing in  (\ref{approx}) one has
$$\norma{x-ry_i}^{-n}=\int_{O(y_i)}\norma{Tx-ry_i}^{-n}\,dm(T),\quad \forall x$$
so 
$$\innerproduct{\rho}{\norma{x-ry_i}^{-n}}=\innerproduct{\rho}{S_{y_i}(\norma{x-ry_i}^{-n})}=\innerproduct{S_{y_i}\rho}{\norma{x-ry_i}^{-n}}\geq 0$$

Hence by continuity $\rho$ is also positive of the r.h.s of (\ref{eq2}) as well. It follows that $\rho$ is a positive distribution.
\end{proof}
It is well known that a positive distribution is in fact a positive measure, that is, it satisfies $|\rho(f)|\leq\rho(1)\norma{f}_{\infty}$ for every test function, where $1$ is the constant function $1$. Consequently, 
there is a unique extension of $\rho$ to a bounded linear functional on $C_e(\sphere{n-1})$, and so in this sense it represents a measure. This completes the proof of the theorem.


\begin{thebibliography}{references}
\bibitem{weil76}
Weil, W, Centrally symmetric convex bodies and distributions, Israel J. Math., 24, 352--367 (1976)
\end{thebibliography}
\end{document}